\documentclass[11pt]{amsart}

\usepackage[T1]{fontenc}
\usepackage[latin1]{inputenc}
\usepackage[english]{babel}
\usepackage{amsmath,amssymb,amsthm} %amssymb loads amsfonts
\usepackage{enumitem}
\usepackage{stmaryrd}
\usepackage{tikz-cd}
\usepackage{hyperref}

\newtheorem{theorem}{Theorem}[section]
\newtheorem{lemma}[theorem]{Lemma}
\newtheorem{proposition}[theorem]{Proposition}
\newtheorem{corollary}[theorem]{Corollary}

\theoremstyle{definition}

\newtheorem{question}[theorem]{Question}
\newtheorem{remark}[theorem]{Remark}

\newtheorem*{remark*}{Remark}

\renewcommand{\Re}{\operatorname{Re}}
\newcommand{\ran}{\operatorname{ran}}

\renewcommand{\MR}[1]{}

\author{Michael Hartz}
\title{Finite dimensional approximations in operator algebras}
\address{Fachrichtung Mathematik, Universit\"at des Saarlandes, 66123 Saarbr\"ucken, Germany}

\email{hartz@math.uni-sb.de}
\thanks{The author was partially supported by a GIF grant and by the Emmy Noether Program of the German Research Foundation (DFG Grant 466012782).}
\date{\today}

\subjclass[2010]{Primary 47L55; Secondary 47L30, 47A20, 47B35}
\keywords{Non self-adjoint operator algebra, residual finite dimensionality, finite dimensional approximation}

\begin{document}

\begin{abstract}
  A non-self-adjoint operator algebra is said to be residually finite dimensional (RFD) if it embeds
  into a product of matrix algebras.
  We characterize RFD operator algebras in terms of their matrix state space, and moreover show
  that an operator algebra is RFD if and only if every representation can be approximated
  by finite dimensional ones in the point weak operator topology.
  This is a non-self-adjoint version of a theorem of Exel and Loring for $C^*$-algebras.
  Moreover, we construct an example of an operator algebra
  for which approximation in the point strong operator topology is not possible.
  As a consequence, the maximal $C^*$-algebra generated by this operator algebra is not RFD.
  This answers questions of Clou\^atre and Ramsey and of Clou\^atre and Dor-On.
\end{abstract}

\maketitle

\section{Introduction}

In this article, an \emph{operator algebra} is a norm closed (not necessarily self-adjoint) subalgebra
$\mathcal{A} \subset B(\mathcal{K})$ for some Hilbert space $\mathcal{K}$. We say that $\mathcal{A}$
is \emph{unital} if the identity operator on $\mathcal{K}$ belongs to $\mathcal{A}$.
By a \emph{representation} of $\mathcal{A}$, we mean a (not necessarily unital) completely contractive homomorphism $\pi: \mathcal{A} \to B(\mathcal{H})$
for some Hilbert space $\mathcal{H}$. Background information on operator algebras and their representations
can be found in \cite{BL04,Paulsen02,Pisier03}.

We will study operator algebras that can be recovered from their finite dimensional representations.
Explicitly, an operator algebra $\mathcal{A}$ is said to be \emph{residually finite dimensional (RFD)} if for every $n \in \mathbb{N}$ and every $a \in M_n(\mathcal{A})$, we have
\begin{equation}
  \label{eqn:norm}
  %\|a\| = \sup \{ \|\pi^{(n)}(a)\|: \pi: \mathcal{A} \to B(\mathcal{H}) \text{ representation with } \dim(\mathcal{H})< \infty\}.
  \|a\| = \sup \{ \|\pi^{(n)}(a)\| \},
\end{equation}
where the supremum is taken over all representations $\pi: \mathcal{A} \to B(\mathcal{H})$ with $\dim(\mathcal{H}) < \infty$.
Here, $\pi^{(n)}(a)$ is the element of $M_n(B(\mathcal{H}))$ obtained by applying $\pi$ to each entry of $a$.
Note that $\mathcal{A}$ is RFD if and only if there exist a family $\{\mathcal{H}_\lambda: \lambda \in \Lambda\}$
of finite dimensional Hilbert spaces and a completely isometric homomorphism
\begin{equation*}
  \pi: \mathcal{A} \to \prod_{\lambda \in \Lambda} B(\mathcal{H}_\lambda).
\end{equation*}

Residual finite dimensionality is a central concept in the theory of $C^*$-algebras.
If $\mathcal{A}$ is a $C^*$-algebra, then contractive homomorphisms $\mathcal{A} \to B(\mathcal{H})$
are automatically $*$-homomorphisms (see, for instance, \cite[Proposition A.5.8]{BL04}),
so for $C^*$-algebras, the notion of residual finite dimensionality considered here agrees with the usual $C^*$-algebraic
notion.
 In the setting of non-self-adjoint operator algebras, residual finite dimensionality was introduced
 by Mittal and Paulsen \cite[Section 3]{MP10} and in particular studied in the context of operator algebras
 of functions.
 More recently, RFD operator algebras have been studied as objects in their own right,
 such as in the works of Clou\^atre and Marcoux \cite{CM19}, Clou\^atre and Ramsey \cite{CR19}, Clou\^atre and Dor-On \cite{CD21} and Thompson \cite{Thompson21}.

 Examples of non-self-adjoint RFD operator algebras include all uniform algebras, all finite dimensional operator algebras
 \cite[Corollary 3.6]{CR19}, all multiplier algebras of reproducing kernel Hilbert spaces \cite{MP10} and
 the algebra of all bounded upper triangular operators on $\ell^2$.
 More involved examples of RFD operator algebras are the Schur--Agler algebra (the algebra
 of holomorphic functions on the polydisc that are bounded on all commuting strict contractions)
 and the Douglas--Paulsen algebra on an annulus; see \cite[Section 5]{MP10} for details.
 
In the context of $C^*$-algebras, a fundamental theorem characterizing residual finite dimensionality
is due to Exel and Loring \cite{EL92}.
If $\mathcal{A}$ is a unital $C^*$-algebra, we let $S_1(\mathcal{A})$ denote the state space of $\mathcal{A}$,
i.e.\ the space of all unital contractive functionals on $\mathcal{A}$, equipped with the weak-$*$ topology.
The GNS construction shows that for every state $\varphi \in S_1(\mathcal{A})$,
there exist a unital $*$-homomorphism $\pi: \mathcal{A} \to B(\mathcal{H})$ and
a unit vector $x \in \mathcal{H}$ with
\begin{equation*}
  \varphi(a) = \langle \pi(a) x,x \rangle \quad \text{ for all } a \in A.
\end{equation*}
The state $\varphi$ is said to be \emph{finite dimensional} if $\mathcal{H}$ can be taken to be finite dimensional.
A representation $\pi: \mathcal{A} \to B(\mathcal{H})$ is said to be \emph{finite dimensional} if
the essential space, which is the closed linear span of $\pi(\mathcal{A}) \mathcal{H}$, is finite dimensional.

\begin{theorem}[Exel--Loring]
  \label{thm:EL_intro}
  The following assertions are equivalent for a unital $C^*$-algebra $\mathcal{A}$:
  \begin{enumerate}[label=\normalfont{(\roman*)}]
    \item $\mathcal{A}$ is RFD;
    \item the set of finite dimensional states is weak-$*$ dense in $S(\mathcal{A})$;
    \item for every representation $\pi: \mathcal{A} \to B(\mathcal{H})$,
      there exists a net $(\pi_\lambda)$ of finite dimensional representations
      such that $(\pi_\lambda(a))$ converges to $\pi(a)$ in SOT for all $a \in \mathcal{A}$.
  \end{enumerate}
\end{theorem}
In fact, Exel and Loring establish the appropriate version of this result in the non-unital setting.
The implication (iii) $\Rightarrow$ (i) in the Exel--Loring theorem follows fairly
immediately by applying (iii) to any faithful representation of $\mathcal{A}$.
Thus, the implication (i) $\Rightarrow$ (iii) can be interpreted as the following statement:
If one faithful representation of $\mathcal{A}$ can be approximated by finite dimensional ones,
then every representation of $\mathcal{A}$ can.

Clou\^atre and Dor-On \cite{CD21} raised the question of whether there is a version of the Exel--Loring
theorem for non-self-adjoint operator algebras.
The first goal of this article to establish such a result.

To state the non-self-adjoint version of the Exel--Loring theorem,
it is convenient to first recall some terminology.
If $\mathcal{A}$ is a unital operator algebra,
a \emph{matrix state} of $\mathcal{A}$ is a unital completely contractive (u.c.c.) linear map
$\varphi: A \to M_n$, where we write $M_n = M_n(\mathbb{C})$. The \emph{matrix state space} of $\mathcal{A}$ 
is the collection $S_n(\mathcal{A}) = (S_n(\mathcal{A}))_{n=1}^\infty$ of sets
\begin{equation*}
  S_n(\mathcal{A}) = \{ \varphi: \mathcal{A} \to M_n: \varphi \text{ is linear and u.c.c.}\}.
\end{equation*}
Since contractive linear maps into $\mathbb{C}$ are automatically completely contractive,
this notation is consistent with the earlier notation $S_1(\mathcal{A})$ for the state space.
We may identify each $S_n(\mathcal{A})$ with a subspace of $M_n(\mathcal{A}')$, where $\mathcal{A}'$ denotes
the dual space of $\mathcal{A}$. We equip $M_n(\mathcal{A}')$ with the product topology of the weak-$*$ topology on $\mathcal{A}'$.

The Arveson extension theorem and the Stinespring dilation theorem imply
that every matrix state $\varphi: \mathcal{A} \to M_n$ dilates to a unital representation of $\mathcal{A}$.
This means that there exist a Hilbert space $\mathcal{H}$, an isometry $w: \mathbb{C}^n \to \mathcal{H}$
and a unital completely contractive homomorphism $\pi: \mathcal{A} \to B(\mathcal{H})$ so that
\begin{equation*}
  \varphi(a) = w^* \pi(a) w \quad \text{ for all } a \in \mathcal{A}.
\end{equation*}
We say that the matrix state $\varphi$ is \emph{finite dimensional} if $\mathcal{H}$ can be chosen to be finite dimensional.

A representation of $\mathcal{A}$ is said to be finite dimensional if the closed linear span of $C^*(\pi(\mathcal{A})) \mathcal{H}$ is finite dimensional. Since $\mathcal{A}$ is assumed to be unital, $\pi(1)$ is a contractive idempotent and hence an orthogonal projection.
Thus, $\pi$ is a finite dimensional representation if and only if $\pi(1) \mathcal{H}$ is finite dimensional.

In Section \ref{sec:EL}, we will establish the following version of the Exel--Loring theorem
for non-self-adjoint operator algebras.

\begin{theorem}
  \label{thm:RFD_char_intro}
  The following assertions are equivalent for a (not necessarily self-adjoint) unital operator algebra $\mathcal{A}$:
  \begin{enumerate}[label=\normalfont{(\roman*)}]
    \item $\mathcal{A}$ is RFD;
    \item the set of finite dimensional matrix states is weak-$*$ dense in the matrix state space $S(\mathcal{A})$;
    \item for every representation $\pi: \mathcal{A} \to B(\mathcal{H})$, there exists a net $(\pi_\lambda)$
      of finite dimensional representations such that $(\pi_\lambda(a))$ converges to $\pi(a)$ in WOT for all $a \in \mathcal{A}$.
  \end{enumerate}
\end{theorem}

The equivalence of (i) and (iii) continues to hold in the non-unital setting; see Corollary \ref{cor:non-unital}.

There are two obvious differences between the Exel--Loring theorem for $C^*$-algebras (Theorem \ref{thm:EL_intro})
and Theorem \ref{thm:RFD_char_intro} for not necessarily self-adjoint algebras.
Firstly, in items (ii), the state space in the $C^*$-case is replaced with the matrix state space in the non-self-adjoint setting. In Section \ref{sec:state_space}, we will show that for a unital operator algebra,
the finite dimensional states are weak-$*$ dense in the state space if and only if Equation \eqref{eqn:norm}
holds for $n=1$, that is, if and only if the norm of $\mathcal{A}$ can be recovered from finite dimensional representations.

Secondly, and perhaps more importantly for the purposes of this article, the approximation by finite dimensional representations
in (iii) is in SOT in the $C^*$-case, but only in WOT in the non-self-adjoint setting.
However, for $C^*$-algebras, the two conditions are the same.
This follows from the simple operator theory fact that for a net $(A_\lambda)$ of bounded
operators, WOT convergence $A_\lambda \to A$ and $A_\lambda^* A_\lambda \to A^* A$ implies SOT convergence
$A_\lambda \to A$.
Thus, the equivalence of (i) and (iii) in Theorem \ref{thm:EL_intro} can be recovered
from the corresponding equivalence in Theorem \ref{thm:RFD_char_intro}.

Nonetheless, the distinction between WOT and SOT turns out to be relevant, and it is natural to ask if WOT convergence in Theorem \ref{thm:RFD_char_intro} can
be improved to SOT convergence.
In fact, this is the non-self-adjoint version of the Exel--Loring theorem
that Clou\^atre and Dor-On asked for; see \cite[Question 2]{CD21}.
They also asked if, even better, approximation in SOT-$*$ is possible, meaning
that $(\pi_\lambda(a))$ converges to $\pi(a)$ in SOT and $(\pi_\lambda(a)^*)$ converges
to $\pi(a)^*$ in SOT for all $a \in \mathcal{A}$; this is again automatic for $C^*$-algebras.
Their question was motivated by considerations regarding the maximal $C^*$-algebra of an operator algebra,
which we will review below.
Clou\^atre and Dor-On also showed that SOT-$*$ approximation is indeed possible for certain classes of operator algebras. For instance, their results include the statement that the universal norm closed operator algebra
generated by $d$ commuting contractions has this approximation property; see \cite[Corollary 5.14]{CD21}.
This has the following operator theory consequence: Every $d$-tuple of commuting contractions on a Hilbert space can be approximated in SOT-$*$ by a $d$-tuple of commuting matrices.

In Section \ref{sec:SOT}, we will show that SOT-approximation, and hence SOT-$*$ approximation,
is not possible for general RFD operator algebras. In some sense,
this also means that it is no accident that the SOT-$*$ approximation results of Clou\^atre and Dor-On
used specific properties of the operator algebras.
Let $\mathbb{D}$ denote the open unit disc in $\mathbb{C}$, let $\mathbb{T} = \partial \mathbb{D}$, and let
\begin{equation*}
  A(\mathbb{D}) = \{f \in C( \overline{\mathbb{D}}): f \big|_{\mathbb{D}} \text{ is holomorphic}\}
\end{equation*}
be the disc algebra. We think of $A(\mathbb{D})$ as a subalgebra of $C(\mathbb{T})$ by virtue of the maximum
modulus principle.

\begin{theorem}
  \label{thm:not_approx_intro}
  Let
  \begin{equation*}
  \mathcal{B} = \left\{
    \begin{bmatrix}
      {f} & 0 \\
      h & \overline{g}
  \end{bmatrix}: f,g \in A(\mathbb{D}), h \in C(\mathbb{T}) \right\} \subset M_2(C(\mathbb{T})).
\end{equation*}
Then:
\begin{enumerate}[label=\normalfont{(\alph*)}]
  \item $\mathcal{B}$ is a unital operator algebra that is RFD;
  \item there exists a unital representation $\pi: \mathcal{B} \to B(\mathcal{H})$
    which is not the point SOT-limit of a net of finite dimensional representations of $\mathcal{B}$.
\end{enumerate}
\end{theorem}

In fact, the representation $\pi$ will take a very concrete form given by Toeplitz operators
on the Hardy space $H^2$, namely
\begin{equation*}
  \pi: \mathcal{B} \to B(H^2 \oplus H^2), \quad
  \begin{bmatrix}
    {f} & 0 \\
    h & \overline{g}
  \end{bmatrix}
  \mapsto
  \begin{bmatrix}
    T_{{f}} & 0 \\
    T_h & T_{\overline{g}}
  \end{bmatrix}.
\end{equation*}
The algebra $\mathcal{B}$ in Theorem \ref{thm:not_approx_intro} is not only RFD,
but it is (completely isometrically) $2$-subhomogeneous, meaning that
it suffices to consider the supremum over all Hilbert spaces of dimension at most $2$ in \eqref{eqn:norm};
see \cite{AHM+20a} for background on this notion.

The question of Clou\^atre and Dor-On on SOT-$*$ approximations was motivated by the following considerations.
Given a non-self-adjoint unital operator algebra $\mathcal{A}$,
there are in general many $C^*$-algebras that are generated by a copy of $\mathcal{A}$.
In particular, there is minimal one, called the $C^*$-envelope, and a maximal one, called
the maximal $C^*$-algebra and denoted by $C^*_{\max}(\mathcal{A})$.
The maximal $C^*$-algebra is characterized by the following universal property:
There exists a unital completely isometric homomorphism $\iota: \mathcal{A} \to C^*_{\max}(\mathcal{A})$
such that $C^*_{\max}(\mathcal{A})$ is generated by $\iota(\mathcal{A})$ as a $C^*$-algebra,
and for every representation $\pi: \mathcal{A} \to B(\mathcal{H})$,
there exists a $*$-homomorphism $\sigma: C^*_{\max}(\mathcal{A}) \to B(\mathcal{H})$
such that $\sigma \circ \iota = \pi$, i.e.\ so that the following diagram commutes:

\begin{center}
\begin{tikzcd}
  & C^*_{\max}(\mathcal{A}) \arrow[d,dashed,"\sigma"] \\
  \mathcal{A} \arrow[ur,"\iota"] \arrow[r,"\pi"] & B(\mathcal{H})
\end{tikzcd}
\end{center}
Background material on the maximal $C^*$-algebra can be found in \cite{Blecher99} and \cite[Section 2.4]{BL04}.
Since residual finite dimensionality is much better understood in the $C^*$-context, it is natural
to try to link residual finite dimensionality of $\mathcal{A}$ to residual finite dimensionality
of a canonical $C^*$-algebra associated with $\mathcal{A}$.

Now, there are simple and natural examples of RFD operator algebras whose $C^*$-envelope
is not RFD. For instance, the $C^*$-envelope of the algebra of all bounded upper triangular operators on $\ell^2$ is $B(\ell^2)$, which has no finite dimensional representations.
Another example is given by Arveson's algebra $\mathcal{A}_d$, which is the multiplier norm
closure of the polynomials on the Drury--Arveson space. It plays a central role in multivariable operator theory; see \cite{Arveson98} for background.
This algebra is RFD, being an algebra of multipliers on a reproducing kernel Hilbert space.
But if $d \ge 2$, then the $C^*$-envelope of $\mathcal{A}_d$ is the Toeplitz $C^*$-algebra,
which contains the algebra of all compact operators and hence is not RFD; see \cite[Theorem 8.15]{Arveson98}.
Clou\^atre and Ramsey even constructed a finite dimensional unital operator algebra whose
$C^*$-envelope is not RFD; see \cite[Section 6, Example 4]{CR19}.

Thus, attention shifted to $C^*_{\max}$, and Clou\^atre and Ramsey \cite{CR19} asked
if $C^*_{\max}(\mathcal{A})$ is RFD whenever $\mathcal{A}$ is RFD.
Since every representation of $\mathcal{A}$ extends to a representation of $C^*_{\max}(\mathcal{A})$ by the universal
property, $C^*_{\max}(\mathcal{A})$ has many finite dimensional representations whenever $\mathcal{A}$ is RFD.
This question was further studied by Clou\^atre and Dor-On \cite[Question 1]{CD21},
who observed that $C^*_{\max}(\mathcal{A})$ is RFD if and only if every representation
of $\mathcal{A}$ can be approximated point SOT-$*$ by finite dimensional ones; see \cite[Theorem 3.3]{CD21}.
Indeed, necessity, which is what we will use, is immediate from the universal property
of $C^*_{\max}(\mathcal{A})$ and the Exel--Loring theorem (Theorem \ref{thm:EL_intro}).

Therefore, Theorem \ref{thm:not_approx_intro} has the following consequence.

\begin{corollary}
  Let $\mathcal{B}$ be the operator algebra of Theorem \ref{thm:not_approx_intro}.
  Then $\mathcal{B}$ is RFD, but $C^*_{\max}(\mathcal{B})$ is not RFD.
\end{corollary}

However, since $\mathcal{B}$ is $2$-subhomogeneous, \cite[Proposition 4.1]{AHM+20a}
shows that the $C^*$-envelope of $\mathcal{B}$
is $2$-subhomogeneous as well and hence RFD.
Indeed, the $C^*$-envelope of $\mathcal{B}$ is $M_2(C(\mathbb{T}))$  since $\mathcal{B}$ is a Dirichlet
algebra; see Remark \ref{rem:extension} and \cite[Proposition 4.3.10]{BL04}.
Now, if $\mathcal{A}$ is any unital RFD operator algebra whose $C^*$-envelope is not RFD, then the direct sum $\mathcal{A} \oplus \mathcal{B}$ is a unital operator algebra with the property that neither the $C^*$-envelope
nor the maximal $C^*$-algebra is RFD. This is because the $C^*$-envelope of a direct sum is the direct sum
of the $C^*$-envelopes, and similarly for the maximal $C^*$-algebra.

\section{A non-self-adjoint Exel--Loring theorem}
\label{sec:EL}

Recall from the introduction that the matrix state space of a unital operator algebra $\mathcal{A}$
is the collection $S(\mathcal{A}) = (S_n(\mathcal{A}))_{n=1}^\infty$ of sets
\begin{equation*}
  S_n(\mathcal{A})  = \{ \varphi: \mathcal{A} \to M_n: \varphi \text{ is linear and u.c.c.}\}.
\end{equation*}
We identify each $S_n(\mathcal{A})$ with a subspace of $M_n(\mathcal{A}')$,
where $\mathcal{A}'$ is the dual space of $\mathcal{A}$.
Then $S(\mathcal{A})$ is a matrix convex set in $\mathcal{A}'$,
meaning that whenever $\varphi_{j} \in S_{n_j}(\mathcal{A})$
and $\alpha_{j} \in M_{n_j,n}(\mathbb{C})$ for $j=1,\ldots,r$ satisfy $\sum_{j=1}^n \alpha_j^* \alpha_j = I_n$,
then $\sum_{j=1}^n \alpha_j^* \varphi_j \alpha_j \in S_n(\mathcal{A})$.
Background on matrix convexity can be found in \cite{Effros1997,WW99}.
We equip $M_n(\mathcal{A}')$ with the product topology of the weak-$*$ topology on $\mathcal{A}'$.
Then each $\mathcal{S}_n(\mathcal{A})$ is compact.

The following result is a slight refinement of Theorem \ref{thm:RFD_char_intro} from the introduction.

\begin{theorem}
  \label{T:EL_nsa}
  Let $\mathcal{A}$ be a unital operator algebra. The following are equivalent:
  \begin{enumerate}[label=\normalfont{(\roman*)}]
    \item $\mathcal{A}$ is RFD;
    \item the finite dimensional matrix states form a weak-$*$ dense subset of the matrix state space $S(\mathcal{A})$;
    \item for every unital representation $\pi$ of $\mathcal{A}$ on $\mathcal{H}$, there exist a net $(\pi_\lambda)$ of finite dimensional representations
      on $\mathcal{H}$ and an increasing net $(P_\lambda)$ of finite rank orthogonal projections
      so that $\pi(a) = \lim_\lambda P_\lambda \pi_\lambda(a) P_\lambda$ in the SOT-$*$ topology
      for all $a \in \mathcal{A}$.
    \item for every unital representation $\pi: \mathcal{A} \to B(\mathcal{H})$, there exists a net $(\pi_\lambda)$
      of finite dimensional representations such that $(\pi_\lambda(a))$ converges to $\pi(a)$ in WOT for all $a \in \mathcal{A}$.
  \end{enumerate}
\end{theorem}

\begin{proof}
  (i) $\Rightarrow$ (ii)
  The proof is essentially an adaptation of the proof of Exel and Loring to the matrix convex setting.
  A very similar argument in the $C^*$-setting already appeared in the first arXiv version of \cite{HL19}.

  Let $F(\mathcal{A}) = (F_n)_{n=1}^\infty$
  denote the collection of all finite dimensional matrix states of $\mathcal{A}$.
  It is a simple matter to verify that $F(\mathcal{A})$ is a matrix convex subset
  of $S(\mathcal{A}) = (S_n(\mathcal{A}))_{n=1}^\infty$.
  Indeed, let $\varphi_j \in F_{n_j}$ and $\alpha_j \in M_{n_j,n}$
  for $j=1,\ldots,r$ satisfy $\sum_{j=1}^n \alpha_j^* \alpha_j = I_n$.
  For each $j$, the map $\varphi_j: \mathcal{A} \to M_{n_j}$ is a finite dimensional matrix state,
  so there exist an isometry $w_j: \mathbb{C}^{n_j} \to \mathcal{H}_j$ and a unital representation
  $\pi_j: \mathcal{A} \to B(\mathcal{H}_j)$, with $\dim(\mathcal{H}_j) < \infty$, so that
  \begin{equation*}
    \varphi_j(a) = w_j^* \pi_j(a) w_j \quad \text{ for all } a \in \mathcal{A}.
  \end{equation*}
  Then the identity
  \begin{equation*}
    \sum_{j=1}^r \alpha_j^* \varphi_j(a) \alpha_j
    =
    \begin{bmatrix}
      \alpha_1^* w_1^* & \cdots & \alpha_r^* w_r^*
    \end{bmatrix}
    \begin{bmatrix}
      \pi_1(a) & 0 & \cdots & 0 \\
      0 & \pi_2(a) & \cdots & 0 \\
      \vdots & \ddots & \ddots & \vdots \\
      0 & 0 & \cdots & \pi_n(a)
    \end{bmatrix}
    \begin{bmatrix}
      w_1 \alpha_1 \\ \vdots \\ w_r \alpha_r
    \end{bmatrix}
  \end{equation*}
  shows that $\sum_{j=1}^r \alpha_j^* \varphi_j \alpha_j \in F_n$.
  
  Assume towards a contradiction that $F(\mathcal{A})$ is not weak-$*$ dense in $S(\mathcal{A})$,
  so that there exist $n \ge 1$ and $\psi \in S_n(\mathcal{A}) \setminus \overline{F_n}^{w*}$.
  We will identify $M_n(\mathcal{A}) = M_n \otimes \mathcal{A}$ in the usual way.
  The separation theorem for matrix convex sets of Effros and Winkler \cite{Effros1997}
  (see \cite[Theorem 1.6]{WW99} for the statement in the form in which it is needed) yields
  a weak-$*$ continuous linear map $\Phi: \mathcal{A}' \to M_n$ and a self-adjoint matrix $\alpha \in M_n$
  such that
  \begin{equation*}
    \Re \Phi^{(r)} (\varphi) \le I_r \otimes \alpha \quad \text{ for all } r \in \mathbb{N}, \varphi \in F_r,
  \end{equation*}
  but
  \begin{equation*}
    \Re \Phi^{(n)}(\psi) \not \le I_n \otimes \alpha.
  \end{equation*}
  Identifying $\Phi$ with an element of $M_n(\mathcal{A}'')$,
  weak-$*$ continuity of $\Phi$ shows that it is given by an element $a \in M_n(\mathcal{A})$.
  Thus, under the canonical
  shuffle (see, for instance, \cite[Chapter 8]{Paulsen02}), the two conditions above translate to
  \begin{equation}
    \label{eqn:HB1}
    \Re \varphi^{(n)}(a) \le \alpha \otimes I_r \quad \text{ for all } r \in \mathbb{N}, \varphi \in F_r,
  \end{equation}
  but
  \begin{equation}
    \label{eqn:HB2}
    \Re \psi^{(n)}(a) \not \le \alpha \otimes I_n.
  \end{equation}

  We will
  show that \eqref{eqn:HB1} implies that $\Re a \le \alpha$ inside of $M_n(\mathcal{A} + \mathcal{A}^*)$.
  Indeed, if $\pi: \mathcal{A} \to M_r$ is a unital finite dimensional representation of $\mathcal{A}$,
  then in particular $\pi \in F_r$, hence
  \begin{equation}
    \label{eqn:HB3}
    \Re \pi^{(n)}(a) \le \alpha \otimes I_r = \pi^{(n)}(\alpha).
  \end{equation}
  Since $\mathcal{A}$ is unital and RFD, the direct sum over a suitable collection of unital finite
  dimensional representations yields a completely isometric representation $\Pi$ of $\mathcal{A}$,
  and $\Re \Pi^{(n)}(a) \le \Pi^{(n)}(\alpha)$ by \eqref{eqn:HB3}.
  Since $\Pi$ is a unital complete isometry, it extends to a unital complete order isomorphism
  between the operator systems $\mathcal{A} + \mathcal{A}^*$ and $\Pi(\mathcal{A}) + \Pi(\mathcal{A})^*$;
  see \cite[Proposition 3.5]{Paulsen02}.
  Hence
  \begin{equation*}
    \Re a \le \alpha \in
    M_n(\mathcal{A} + \mathcal{A}^*).
  \end{equation*}

  On the other hand, the u.c.c.\ map $\psi: \mathcal{A} \to M_n$ extends to a u.c.p.\ map from $\mathcal{A} + \mathcal{A}^*$ into $M_n$, so
  \begin{equation*}
    \Re \psi^{(n)}(a) \le \psi^{(n)}(\alpha) = \alpha \otimes I_n,
  \end{equation*}
  contradicting \eqref{eqn:HB2}.

  (ii) $\Rightarrow$ (iii)
  Let $\pi: \mathcal{A} \to B(\mathcal{H})$ be a unital representation. Clearly, we may assume that $\mathcal{H}$ is infinite dimensional.
  Let $(P_\lambda)_{\lambda \in \Lambda}$ be an increasing net of finite rank projections converging to the identity in SOT.
  Then 
  \begin{equation}
    \label{eqn:1}
    \pi(a) = \lim_{\lambda} P_\lambda \pi(a) P_\lambda \text{ in SOT-$*$} \text{ for all } a \in \mathcal{A}.
  \end{equation}
  For each $\lambda$, we may regard $a \mapsto P_\lambda \pi(a) P_\lambda$ as a matrix state on $\mathcal{A}$.
  By assumption, there exists for every finite subset $F \subset A$, every $\lambda \in \Lambda$
  and every $k \in \mathbb{N}$
  a finite dimensional
  matrix state $\varphi_{F,\lambda,k}: \mathcal{A} \to B(P_\lambda \mathcal{H})$ so that
  \begin{equation*}
    \|P_\lambda \pi(a) P_\lambda - \varphi_{F,\lambda,k}(a) \| < \frac{1}{k}
  \end{equation*}
  for all $a \in F$. Each $\varphi_{F,\lambda,k}$ dilates to a finite dimensional representation $\pi_{F,\lambda,k}$ of $\mathcal{A}$.
  Since $\mathcal{H}$ is infinite dimensional, we may assume that the range of $\pi_{F,\lambda,k}$ is contained in $\mathcal{H}$.
  Thus,
  \begin{equation*}
    \|P_\lambda \pi(a) P_\lambda - P_\lambda \pi_{F,\lambda,k}(a) P_\lambda\| < \frac{1}{k}
  \end{equation*}
  for all $a \in F$. In combination with \eqref{eqn:1}, we see that the net $(P_\lambda \pi_{F,\lambda,k}(a) P_\lambda)$, indexed over $\mathcal{F}(\mathcal{A}) \times \Lambda \times \mathbb{N}$ in the product order,
  where $\mathcal{F}(\mathcal{A})$ denotes the set of all finite subsets of $\mathcal{A}$,
  converges
  to $\pi(a)$ in SOT-$*$ for all $a \in \mathcal{A}$.

  (iii) $\Rightarrow$ (iv) Let $\pi_\lambda$ and $P_\lambda$ be as in (iii).
  Then
  \begin{equation*}
    \pi_\lambda(a) - P_\lambda \pi_\lambda(a) P_\lambda
    = (I - P_\lambda)\pi_\lambda(a) + P_\lambda \pi_\lambda(a) (I - P_\lambda).
  \end{equation*}
  Since $\|\pi_\lambda(a)\| \le \|a\|$ and $(I - P_\lambda)$ converges to zero in SOT,
  both summands converge to zero in WOT. Hence, in the setting of (iii), we have
  $\pi(a) = \lim_\lambda \pi_\lambda(a)$ in WOT.

  (iv) $\Rightarrow$ (i)
  We apply (iv) to a completely isometric representation $\pi$ of $\mathcal{A}$.
  Since the operator norm is lower semi-continuous with respect to WOT, we see that for all $a \in M_n(\mathcal{A})$,
  we have
  \begin{equation*}
    \|a\| = \sup_{\lambda \in \Lambda} \|\pi_\lambda^{(n)}(a)\|,
  \end{equation*}
  hence $\mathcal{A}$ is RFD.
\end{proof}

The net $(\pi_\lambda)$ in part (iv) of Theorem \ref{T:EL_nsa} converges to $\pi$ in the BW topology, a frequently
studied topology in operator algebras; see for instance \cite[Chapter 7]{Paulsen02}.

\begin{remark}
    If $\mathcal{A}$ and $\mathcal{H}$ are separable, then the nets in Theorem \ref{T:EL_nsa} can be replaced with sequences.
      To see this, let $D \subset A$ be a countable dense subset. A straightforward modification of the proof of (ii) $\Rightarrow$ (iii) above
      yields an increasing sequence $(P_n)$ of orthogonal projections and a sequence $(\pi_n)$ of finite dimensional representations of $\mathcal{A}$
      such that
      \begin{equation*}
        \pi(a) = \lim_{n \to \infty} P_n \pi_n(a) P_n \quad \text{ in SOT-$*$}
      \end{equation*}
      for all $a \in D$. Since $\|\pi_n\| \le 1$ and $\|P_n\| \le 1$ for all $n$, we obtain convergence for all $a \in A$.
\end{remark}

\begin{remark}
  The implication (i) $\Rightarrow$ (ii) in Theorem \ref{T:EL_nsa} could alternatively be deduced from a corresponding
  result for $C^*$-algebras.
  According to \cite[Theorem 3.8]{ANT19}, the equivalence (i) $\Leftrightarrow$ (ii) in Theorem \ref{T:EL_nsa} holds
  if $\mathcal{A}$ is a unital $C^*$-algebra.
  Now, if $\mathcal{A}$ is a unital operator algebra that is RFD, then we may embed $\mathcal{A}$
  into a $C^*$-algebra of the form $\mathfrak{A} = \prod_{\lambda \in \Lambda} B(\mathcal{H}_\lambda)$,
  where each $\mathcal{H}_\lambda$ is finite dimensional. Then $\mathfrak{A}$ is a unital $C^*$-algebra
  that is RFD.
  By the Arveson extension theorem, each matrix state on $\mathcal{A}$ extends to a matrix state on $\mathfrak{A}$.
  Conversely, finite dimensional matrix states on $\mathfrak{A}$ restrict to finite dimensional matrix states on $\mathcal{A}$.
  Therefore, density of the finite dimensional matrix states on $\mathfrak{A}$ implies
  density of the finite dimensional matrix states on $\mathcal{A}$.
\end{remark}

We observe that the equivalence of (i), (iii) and (iv) in Theorem \ref{T:EL_nsa} continues to hold
in the non-unital setting. If $\mathcal{A}$ is an operator algebra, then a representation $\pi: \mathcal{A} \to B(\mathcal{H})$ is said to be finite dimensional if the closed linear span of $C^*(\pi(\mathcal{A})) \mathcal{H}$ is finite dimensional.
(This is stronger than merely assuming that the closed linear span of $\pi(\mathcal{A}) \mathcal{H}$ is finite dimensional in the non-unital setting, as simple examples of operator algebras with zero product show.)

\begin{corollary}
  \label{cor:non-unital}
  Let $\mathcal{A}$ be a (not necessarily unital) operator algebra. The following are equivalent:
  \begin{enumerate}[label=\normalfont{(\roman*)}]
    \item $\mathcal{A}$ is RFD;
    \item for every representation $\pi$ of $\mathcal{A}$ on $\mathcal{H}$, there exist a net $(\pi_\lambda)$ of finite dimensional representations
      on $\mathcal{H}$ and an increasing net $(P_\lambda)$ of finite rank orthogonal projections
      so that $\pi(a) = \lim_\lambda P_\lambda \pi_\lambda(a) P_\lambda$ in the SOT-$*$ topology
      for all $a \in \mathcal{A}$.
    \item for every representation $\pi: \mathcal{A} \to B(\mathcal{H})$, there exists a net $(\pi_\lambda)$
      of finite dimensional representations such that $(\pi_\lambda(a))$ converges to $\pi(a)$ in WOT for all $a \in \mathcal{A}$.
  \end{enumerate}
\end{corollary}

\begin{proof}
  (i) $\Rightarrow$ (ii)
  It is a theorem of Meyer \cite{Meyer01} that every operator algebra $\mathcal{A}$ admits
  a unitization ${\mathcal{A}^1}$; see also \cite[Section 2.1]{BL04}.
  If $\mathcal{A}$ is RFD, then $\mathcal{A}$ embeds completely isometrically into an algebra of the form $\prod_{\lambda \in \Lambda} B(\mathcal{H}_\lambda)$, where each $\mathcal{H}_\lambda$ is finite dimensional.
  By the universal property of the unitization (see \cite[Corollary 2.1.15]{BL04}), this embedding extends to a unital completely isometric homomorphism ${A}^1 \to \prod_{\lambda \in \Lambda} B(\mathcal{H}_\lambda)$; thus
  ${\mathcal{A}}^1$ is RFD as well.
  Every representation of $\mathcal{A}$ extends to a unital representation of $\mathcal{A}^1$ (see \cite[Theorem 2.1.13]{BL04}). Moreover, finite dimensional representations of $\mathcal{A}^1$ restrict to finite dimensional representations of $\mathcal{A}$. Thus, the implication (i) $\Rightarrow$ (iii) of Theorem \ref{T:EL_nsa},
  applied to the unitization $\mathcal{A}^1$, yields the approximation statement in the non-unital case.
  
  The implications (ii) $\Rightarrow$ (iii) $\Rightarrow$ (i) are immediate, just as in the unital setting.
\end{proof}

\section{Density of finite dimensional states}
\label{sec:state_space}

Whereas the main focus of this article lies on finite dimensional approximations of representations, it is natural to ask
what happens if we replace the matrix state space with the ordinary state space in Theorem \ref{thm:RFD_char_intro}.
That is, which (non-self-adjoint) operator algebras have the property that the set of finite dimensional states is weak-$*$ dense
in the state space? A characterization is given by the following result.

\begin{proposition}
  \label{prop:state_space}
  The following assertions are equivalent for a unital operator algebra $\mathcal{A}$:
  \begin{enumerate}[label=\normalfont{(\roman*)}]
    \item The set of finite dimensional states is weak-$*$ dense in the state space $S_1(\mathcal{A})$;
    \item for every $a \in \mathcal{A}$, we have
      \begin{equation*}
      \|a\| = \sup \{ \|\pi(a)\| \},
      \end{equation*}
      where the supremum is taken over all completely contractive homomorphisms $\pi: \mathcal{A} \to B(\mathcal{H})$ with $\dim(\mathcal{H}) < \infty$.
  \end{enumerate}
\end{proposition}

Clearly, (ii) may also be rephrased by saying that there exist a family $\{\mathcal{H}_\lambda: \lambda \in \Lambda\}$ of finite dimensional Hilbert spaces and a completely contractive, isometric homomorphism $\pi: \mathcal{A} \to \prod_{\lambda \in \Lambda} B(\mathcal{H}_\lambda)$.

In the proof of Proposition \ref{prop:state_space}, we will use the Cayley transform.
The following lemma is certainly well known and essentially appears already in \cite[\S 5.2]{Neumann51},
see also \cite[Lemma 2.1]{Meyer01}. Since the exact formulation we require is slightly different,
we sketch the proof.

 \begin{lemma}
   \label{lem:cayley}
   The maps
   \begin{align*}
     \{ T \in B(\mathcal{H}): -1 \notin \sigma(T) \} &\leftrightarrow \{ S \in B(\mathcal{H}): 1 \notin \sigma(S) \} \\
     T &\mapsto (T-I)(T+I)^{-1} \\
     (I+S)(I-S)^{-1} &\mapsfrom S
   \end{align*}
   are mutually inverse bijections that restrict to bijections
   \begin{equation*}
     \{ T \in B(\mathcal{H}): \Re T \ge 0 \} \leftrightarrow \{ S \in B(\mathcal{H}): \|S\| \le 1, 1 \notin \sigma(S) \}.
   \end{equation*}
   
 \end{lemma}

 \begin{proof}
   The spectral mapping theorem shows that if $-1 \notin \sigma(T)$ and $S = (T-I)(T+I)^{-1}$, then $1 \notin \sigma(S)$. Similarly, if $1 \notin \sigma(S)$ and $T = (I+S)(I-S)^{-1}$, then $-1 \notin \sigma(T)$.
   A routine computation then shows that the two maps in the statement are inverse to each other.

   Next, if $T \in B(\mathcal{H})$ with $\Re T \ge 0$, then the numerical range and hence the spectrum of $T$ is contained
   in the closed right half plane, and so in particular $-1 \notin \sigma(T)$.
   Finally, let $T \in B(\mathcal{H})$ with $-1 \notin \sigma(T)$
   and $S = (T-I)(T+I)^{-1}$. Then
   \begin{equation*}
     \begin{split}
     I - S^* S &= I - (T^* + I)^{-1} (T^* - I) (T - I) (T + I)^{-1} \\
     &= 2 (T^* + I)^{-1} ( T + T^*) (T+I)^{-1}.
     \end{split}
   \end{equation*}
   This operator is positive if and only if $T + T^* \ge 0$. Therefore, $\|S\| \le 1$
   if and only if $\Re T \ge 0$.
 \end{proof}

 \begin{proof}[Proof of Proposition \ref{prop:state_space}]
   (i) $\Rightarrow$ (ii). We regard $\mathcal{A} \subset B(\mathcal{K})$ for some Hilbert space $\mathcal{K}$.
   Let $a \in \mathcal{A}$ and suppose that $\sup \{ \|\pi(a)\| \} = r < 1$, where
  the supremum is taken over the maps $\pi$ occurring in (ii). We will show that $\|a\| \le 1$. This will establish (ii).

  Let $\varphi \in S_1(\mathcal{A})$ be a finite dimensional state. Then $\varphi$ dilates to a finite dimensional representation of $\mathcal{A}$,
  and so $|\varphi(a)| \le r$ by assumption on $a$. Since the finite dimensional states are dense in ${S}_1(\mathcal{A})$, we have
  $|\varphi(a)| \le r$ for all  $\varphi \in S_1(\mathcal{A})$.
  It follows that the numerical radius of $a$, regarded as an operator on $\mathcal{K}$, is at most $r$, and so the spectral radius
  of $a$ is at most $r$. In particular, $1 - a$ is invertible in $\mathcal{A}$.

  Let $b= (1+a)(1-a)^{-1} \in \mathcal{A}$. We will show that $\Re b \ge 0$.
  If $\pi$ is a unital finite dimensional representation of $\mathcal{A}$, then $\|\pi(a)\| \le r < 1$,
  so Lemma \ref{lem:cayley} implies that
  \begin{equation*}
    \Re \pi(b) = \Re [ (I + \pi(a))(I - \pi(a))^{-1} ] \ge 0.
  \end{equation*}
Let $\varphi \in S_1(\mathcal{A})$ be a finite dimensional state. Then there exists a unital finite dimensional representation $\pi: \mathcal{A} \to B(\mathcal{H})$ and a unit vector $\xi \in \mathcal{H}$ with $\varphi = \langle \pi(\cdot) \xi, \xi \rangle$. Thus,
\begin{equation*}
  \Re \varphi(b) = \Re \langle \pi(b) \xi, \xi \rangle \ge 0.
\end{equation*}
Density of the finite dimensional states yields that $\Re \varphi(b) \ge 0$ for all states $\varphi \in S_1(\mathcal{A})$, and so $\Re b \ge 0$ as an element of $B(\mathcal{K})$. Applying Lemma \ref{lem:cayley} again,
we find that $\|a\| \le 1$, as desired.

(ii) $\Rightarrow$ (i) By assumption, there exists a unital completely contractive isometric homomorphism
$\pi: \mathcal{A} \to \prod_{\lambda \in \Lambda} B(\mathcal{H}_\lambda)$, where $\dim(\mathcal{H}_\lambda) < \infty$ for all $\lambda \in \Lambda$. Since $\pi$ is a unital isometry, we obtain a homeomorphism
\begin{equation*}
  \Phi: S_1(\pi(\mathcal{A})) \to S_1(\mathcal{A}), \quad \varphi \mapsto \varphi \circ \pi.
\end{equation*}
Since $\pi$ is completely contractive, every representation of $\pi(\mathcal{A})$ induces a representation of $\mathcal{A}$, and so
$\Phi$ maps finite dimensional states in $S_1(\pi(\mathcal{A}))$ to finite dimensional states in $S_1(\mathcal{A})$.
Finally, $\pi(\mathcal{A})$ is RFD since $\prod_{\lambda \in \Lambda} B(\mathcal{H}_\lambda)$ is RFD,
so the finite dimensional states in $S_1(\pi(\mathcal{A}))$ are weak-$*$ dense in $S_1(\pi(\mathcal{A}))$
by Theorem \ref{T:EL_nsa}. Therefore, the finite dimensional states in $S_1(\mathcal{A})$ are weak-$*$ dense in $S_1(\mathcal{A})$.
\end{proof}

Proposition \ref{prop:state_space} shows that if the finite dimensional states are weak-$*$ dense in the state
space of $\mathcal{A}$, then the norm of $\mathcal{A}$ can be recovered from finite dimensional representations
of $\mathcal{A}$. This raises the question of whether the norms of $M_n(\mathcal{A})$ can also be recovered,
i.e.\ whether $\mathcal{A}$ is RFD in this case.

Conversely,
does there exist a non-RFD operator algebra $\mathcal{A}$ whose finite dimensional states
are weak-$*$ dense in the state space?
Let us call an algebra of the form $\prod_{\lambda \in \Lambda} B(\mathcal{H}_\lambda)$, where $\dim(\mathcal{H}_\lambda) < \infty$ for all $\lambda \in \Lambda$, a product of matrix algebras.
With this terminology, Proposition \ref{prop:state_space} shows that we may equivalently ask:

\begin{question}
  Does there exist a unital operator algebra $\mathcal{A}$ that admits a unital completely contractive
  isometric embedding into a product of matrix algebras, but not a unital completely isometric embedding?
\end{question}

\section{SOT approximation}
\label{sec:SOT}

In this section, we will prove Theorem \ref{thm:not_approx_intro}.
As in the introduction, let
\begin{equation*}
  \mathcal{B} = \left\{
    \begin{bmatrix}
      {f} & 0 \\
      h & \overline{g}
  \end{bmatrix}: f,g \in A(\mathbb{D}), h \in C(\mathbb{T}) \right\} \subset M_2(C(\mathbb{T})).
\end{equation*}
It is clear that $\mathcal{B}$ is a unital operator algebra that is RFD; in fact, it is (completely isometrically)
$2$-subhomogeneous; see \cite{AHM+20a}.

To exhibit a representation of $\mathcal{B}$ that cannot be approximated by finite dimensional representations in SOT,
we have to recall a few basic facts about Toeplitz operators.
Let $H^2 \subset L^2(\mathbb{T})$ denote the classical Hardy space,
which consists of all functions in $L^2(\mathbb{T})$ whose negative Fourier coefficients vanish.
The Hardy space can also be thought of as a space of holomorphic functions on $\mathbb{D}$, but we will exclusively
work with the description as a subspace of $L^2(\mathbb{T})$.
Let $P: L^2(\mathbb{T}) \to H^2$ denote the orthogonal projection. If $h \in C(\mathbb{T})$, then
the \emph{Toeplitz operator} with symbol $h$ is defined to be
\begin{equation*}
  T_h: H^2 \to H^2, \quad f \mapsto P(h f).
\end{equation*}
It is clear that $T_h$ is a bounded linear operator with $\|T_h\| \le \|h\|_\infty$;
in fact, equality holds.
For background material on Toeplitz operators, see for instance \cite[Chapter 7]{Douglas98}.

Let
\begin{equation*}
  \pi: \mathcal{B} \to B(H^2 \oplus H^2), \quad
  \begin{bmatrix}
    {f} & 0 \\
    h & \overline{g}
  \end{bmatrix}
  \mapsto
  \begin{bmatrix}
    T_{{f}} & 0 \\
    T_h & T_{\overline{g}}
  \end{bmatrix}.
\end{equation*}
We call $\pi$ the \emph{Toeplitz representation} of $\mathcal{B}$.

\begin{proposition}
  The map $\pi$ is a unital completely contractive homomorphism and hence a representation of $\mathcal{B}$.
\end{proposition}

\begin{proof}
  If $f_i,g_i \in A(\mathbb{D})$ and $h_i \in C(\mathbb{T})$ for $i=1,2$, then
  \begin{equation*}
    \begin{bmatrix}
      f_1 & 0 \\
      h_1 & \overline{g_1}
    \end{bmatrix}
    \begin{bmatrix}
      f_2 & 0 \\
      h_2 & \overline{g_2}
    \end{bmatrix}
    =
    \begin{bmatrix}
      f_1 f_2 & 0 \\
      h_1 f_2 + \overline{g_1} h_2 & \overline{g_1} \overline{g_2}.
    \end{bmatrix}
  \end{equation*}
  Therefore, the basic relations
  \begin{equation*}
    T_{\overline{g}} T_h = T_{\overline{g} h} \quad \text{ and } \quad T_{h} T_f = T_{h f}
  \end{equation*}
  for $h \in C(\mathbb{T})$ and $f,g \in A(\mathbb{D})$, see \cite[Proposition 7.5]{Douglas98}, imply that $\pi$ is multiplicative.
  
  To see that $\pi$ is completely contractive, we will construct a dilation to a $*$-homomorphism.
  For $h \in C(\mathbb{T})$ let
  \begin{equation*}
    M_h: L^2(\mathbb{T}) \to L^2(\mathbb{T}), \quad f \mapsto h f,
  \end{equation*}
  denote the multiplication operator with symbol $h$.
  Let
  \begin{equation*}
    \sigma: M_2(C(\mathbb{T})) \to B(L^2(\mathbb{T}) \oplus L^2(\mathbb{T})), \quad
    \begin{bmatrix}
      h_{11} & h_{12} \\ h_{21} & h_{22}
    \end{bmatrix}
    \mapsto
    \begin{bmatrix}
      M_{h_{11}} & M_{h_{12}} \\
      M_{h_{21}} & M_{h_{22}}
    \end{bmatrix}.
  \end{equation*}
  Then $\sigma$ is a $*$-homomorphism, and hence completely contractive.
  Moreover,
  \begin{equation*}
    \pi(a) = (P \oplus P) \sigma(a) \big|_{H^2 \oplus H^2} \quad \text{ for all } a \in \mathcal{B},
  \end{equation*}
  so $\pi$ is completely contractive as well.

  (Multiplicativity of $\pi$ can also be seen from this dilation,
  since $H^2 \oplus L^2$ and $(H^2 \oplus L^2) \ominus (H^2 \oplus H^2) = 0 \oplus (H^2)^\bot$ are invariant for $\sigma(\mathcal{B})$,
  so $H^2 \oplus H^2$ is semi-invariant.)
\end{proof}

\begin{remark}
  \label{rem:extension}
  Let $\mathcal{B}^* = \{ B^*: B \in \mathcal{B}\} \subset M_2(C(\mathbb{T}))$. Then $\mathcal{B} + \mathcal{B}^*$ is dense in $M_2(C(\mathbb{T}))$ as $A(\mathbb{D}) + A(\mathbb{D})^*$ is dense in $C(\mathbb{T})$. By definition, this means that $\mathcal{B}$ is a Dirichlet algebra.

  As a consequence, we can see that the Toeplitz representation $\pi$ does not extend to a $*$-representation of $M_2(C(\mathbb{T}))$.
  Indeed, since $\mathcal{B} + \mathcal{B}^*$ is dense in $M_2(C(\mathbb{T}))$, the unital complete contraction $\pi$ has a unique
  extension to a completely positive map on $M_2(C(\mathbb{T}))$, given by
  \begin{equation*}
    M_2(C(\mathbb{T})) \to B(H^2 \oplus H^2), \quad
    \begin{bmatrix}
      a & b \\ c & d
    \end{bmatrix}
    \mapsto
    \begin{bmatrix}
      T_a & T_b \\ T_c & T_d
    \end{bmatrix},
  \end{equation*}
  and this map is not multiplicative as the map $C(\mathbb{T}) \to B(H^2), h \mapsto T_h$, is not multiplicative.
\end{remark}

If $\pi$ did extend to a $*$-representation of $M_2(C(\mathbb{T}))$, then it would be the point SOT limit of a net of finite dimensional representations by the Exel--Loring theorem. Thus, Remark \ref{rem:extension} can be regarded as a consistency check for the following result,
which will complete the proof of Theorem \ref{thm:not_approx_intro}.
\begin{theorem}
  \label{thm:Toeplitz_no_approx}
  The Toeplitz representation $\pi$ is not the point SOT limit of a net of finite dimensional representations of $\mathcal{B}$.
\end{theorem}

Roughly speaking, the fact that the Toeplitz representation cannot be approximated point SOT by finite dimensional
representations is related to the fact that
every isometry or co-isometry on a finite dimensional Hilbert space
is unitary, but the operator
\begin{equation*}
  \pi \left(
    \begin{bmatrix}
      z & 0 \\
      0 & \overline{z}
  \end{bmatrix} \right)
    =
    \begin{bmatrix}
      T_z & 0 \\
      0 & T_{\overline{z}}
    \end{bmatrix}
\end{equation*}
cannot be approximated in SOT by unitary operators, since it is not an isometry.
(But it can be approximated in WOT.)
The main difficulty here is that the algebra $\mathcal{B}$ admits
many representations $\sigma$ for which
\begin{equation*}
  \sigma \left(
    \begin{bmatrix}
      z & 0 \\
      0 & \overline{z}
  \end{bmatrix} \right)
\end{equation*}
is not built from isometries or co-isometries. Indeed, for any pair of contractions $T$ and $S$
on possibly different Hilbert spaces $\mathcal{H}$ and $\mathcal{K}$,
von Neumann's inequality (or the Sz.-Nagy dilation theorem) show that there exists
a representation $\sigma$ of $\mathcal{B}$ on $B(\mathcal{H} \oplus \mathcal{K})$ with
\begin{equation*}
  \sigma \left(
  \begin{bmatrix}
    f & 0 \\
    h & \overline{g}
\end{bmatrix} \right)
  =
  \begin{bmatrix}
    f(T) & 0 \\
    0 & g(S)^*
  \end{bmatrix}
\end{equation*}
for all $h \in C(\mathbb{T})$ and all $f,g \in \mathbb{C}[z]$.
We will need to separate out these representations, which we will call \emph{diagonal type}.

First, we show that any representation of $\mathcal{B}$ decomposes in matrix form.

\begin{lemma}
  \label{L:rep_dec}
  Let $\sigma: \mathcal{B} \to B(\mathcal{H})$ be a unital representation.
  Then there exist an orthogonal decomposition $\mathcal{H} = \mathcal{H}_1 \oplus \mathcal{H}_2$,
  unital representations $\sigma_j: A(\mathbb{D}) \to B(\mathcal{H}_j)$ for $j=1,2$
  and a completely contractive linear map $\varphi: C(\mathbb{T}) \to B(\mathcal{H}_1,\mathcal{H}_2)$
  such that
  \begin{equation*}
    \sigma
    \left(
      \begin{bmatrix}
        f & 0 \\ h & \overline{g}
      \end{bmatrix} \right)
      = \begin{bmatrix}
       \sigma_1(f) & 0 \\ \varphi(h) & \sigma_2(g)^*
      \end{bmatrix}
  \end{equation*}
  with respect to the decomposition $\mathcal{H} = \mathcal{H}_1 \oplus \mathcal{H}_2$.
  Such a decomposition of $\sigma$ is unique.
  Moreover,
  \begin{equation*}
    \varphi(\overline{g} h f) = \sigma_2(g)^* \varphi(h) \sigma_1(f)
  \end{equation*}
  for all $f,g \in A(\mathbb{D})$ and all $h \in C(\mathbb{T})$.
\end{lemma}

\begin{proof}
  This follows from routine arguments involving $2 \times 2$ matrices.
  For completeness, we provide the proof.
  Let
  \begin{equation*}
    P_1 =
    \sigma \left(
      \begin{bmatrix}
        1 & 0 \\ 0 & 0
    \end{bmatrix} \right)
    \quad \text{ and }
    \quad
    P_2 =
    \sigma \left(
      \begin{bmatrix}
        0 & 0 \\ 0 & 1
    \end{bmatrix} \right).
  \end{equation*}
  Then $P_1$ and $P_2$ are contractive idempotents and hence orthogonal projections satisfying
  $P_1 P_2 = 0 = P_2 P_1$ and $P_1 + P_2 = I$.
  Thus, if we define $\mathcal{H}_j = \ran(P_j)$ for $j=1,2$, then
  we obtain an orthogonal decomposition $\mathcal{H} = \mathcal{H}_1 \oplus \mathcal{H}_2$.

  If $f \in A(\mathbb{D})$, then
  \begin{equation*}
    \begin{bmatrix}
      f & 0 \\
      0 & 0
    \end{bmatrix}
    =
    \begin{bmatrix}
      1 & 0 \\
      0 & 0
    \end{bmatrix}
    \begin{bmatrix}
      f & 0 \\
      0 & 0
    \end{bmatrix}
    =
    \begin{bmatrix}
      f & 0 \\
      0 & 0
    \end{bmatrix}
    \begin{bmatrix}
      1 & 0 \\
      0 & 0
    \end{bmatrix}.
  \end{equation*}
  Applying $\sigma$ to these equations, it follows that
  \begin{equation*}
    \sigma \left(
      \begin{bmatrix}
        f & 0 \\
        0 & 0
      \end{bmatrix}
    \right) =
    \begin{bmatrix}
      \sigma_1(f) & 0 \\
      0 & 0
    \end{bmatrix}
  \end{equation*}
  for some map $\sigma_1: A(\mathbb{D}) \to B(\mathcal{H}_1)$, which is necessarily
  a unital representation.
Similarly, there exists a representation $\widetilde{\sigma}: {A(\mathbb{D})}^* \to B(\mathcal{H}_2)$
  such that
  \begin{equation*}
    \sigma \left(
      \begin{bmatrix}
        0 & 0 \\
        0 & \overline{g}
      \end{bmatrix}
    \right) =
    \begin{bmatrix}
      0 & 0 \\
    0 & \widetilde{\sigma}_2(\overline{g})
    \end{bmatrix}
  \end{equation*}
  Defining $\sigma_2(g) = \widetilde{\sigma}_2(\overline{g})^*$, we obtain the desired representation
  $\sigma_2: A(\mathbb{D}) \to B(\mathcal{H}_2)$.

  Finally, if $h \in C(\mathbb{T})$ and $f,g \in A(\mathbb{D})$, then
  \begin{equation*}
    \begin{bmatrix}
      0 & 0 \\
      \overline{g} h f & 0
    \end{bmatrix}
    =
    \begin{bmatrix}
      0 & 0 \\
      0 & \overline{g}
    \end{bmatrix}
    \begin{bmatrix}
      0 & 0 \\
      h & 0
    \end{bmatrix}
    \begin{bmatrix}
      f & 0 \\
      0 & 0
    \end{bmatrix}.
  \end{equation*}
  Firstly, choosing $f=g = 1$ and applying $\sigma$, we find that
  \begin{equation*}
    \sigma \left(
      \begin{bmatrix}
        0 & 0 \\
        h & 0
      \end{bmatrix}
    \right) =
    \begin{bmatrix}
      0 & 0 \\
      \varphi(h) & 0
    \end{bmatrix}
  \end{equation*}
  for some necessary completely contractive linear map $\varphi: C(\mathbb{T}) \to B(\mathcal{H}_1,\mathcal{H}_2)$.
  Secondly, it then follows that
  \begin{equation*}
    \varphi(\overline{g} h f) = \sigma_2(g)^* \varphi(h) \sigma_1(f)
  \end{equation*}
  for all $h \in C(\mathbb{T})$ and all $f,g \in A(\mathbb{D})$, which completes the proof of
  existence and of the additional statement.

  To see uniqueness, simply observe that if we are given any such decomposition of $\sigma$,
  then the orthogonal projection onto $\mathcal{H}_1$ is necessarily given by
  \begin{equation*}
    \begin{bmatrix}
      1 & 0 \\ 0 & 0
    \end{bmatrix}
    = \sigma \left(
    \begin{bmatrix}
      1 & 0 \\ 0 & 0
  \end{bmatrix} \right),
  \end{equation*}
  and similarly for $\mathcal{H}_2$. The maps $\sigma_1,\sigma_2$ and $\varphi$
  are then uniquely determined by the orthogonal projections and by $\sigma$.
\end{proof}

We require the following standard fact from operator theory.
\begin{lemma}
  \label{lem:unitary_cnu_dec}
  Let $\mathcal{H}$ be a finite dimensional Hilbert space and let
  $A \in B(\mathcal{H})$ with $\|A\| \le 1$. Then $A$ decomposes as $A = A_0 \oplus U$,
  where $U$ is unitary and $\lim_{n \to \infty} A_0^n = 0$.
\end{lemma}

\begin{proof}
  This is the decomposition of a contraction into unitary and completely non-unitary part,
  see \cite[Theorem I.3.2]{SFB+10}, specialized to finite dimensions.
  For the convenience of the reader, we provide a linear algebra argument.

  We may without loss of generality assume that $A$ is an upper triangular matrix.
  If the $k$-th diagonal entry of $A$ has modulus one, then since $\|A\| \le 1$,
  the $k$-th row and the $k$-th column are zero outside of the diagonal.
  By permuting the basis, we can therefore write
  \begin{equation*}
    A =
    \begin{bmatrix}
      U & 0 \\
      0 & A_0
    \end{bmatrix},
  \end{equation*}
  where $U$ is a diagonal unitary matrix and $A_0$ is upper triangular with diagonal entries
  of modulus strictly less than one. Hence the spectral radius of $A_0$ is strictly less than one,
  so that $\lim_{n \to \infty} A_0^n = 0$.
\end{proof}

Let $\sigma: \mathcal{B} \to B(\mathcal{H})$ be a unital representation,
decomposed as in Lemma \ref{L:rep_dec}. We say that $\sigma$
is of \emph{unitary type} if $\sigma_1(z)$ and $\sigma_2(z)$ are unitary operators on $\mathcal{H}_1$
and $\mathcal{H}_2$, respectively.
We say that $\sigma$ is of \emph{diagonal type} if $\varphi = 0$.

\begin{lemma}
  \label{lem:rep_dec}
  Let $\sigma: \mathcal{B} \to B(\mathcal{H})$ be a unital representation and assume that $\dim \mathcal{H} < \infty$.
  Then there exists a reducing subspace $M$ for $\sigma$ such that
  \begin{equation*}
    \sigma_M: \mathcal{B} \to B(M), \quad a \mapsto \sigma(a) \big|_M,
  \end{equation*}
  is of unitary type and
  \begin{equation*}
    \sigma_{M^\bot}: \mathcal{B} \to B(M^\bot), \quad a \mapsto \sigma(a) \big|_{M^\bot}
  \end{equation*}
  is of diagonal type.
\end{lemma}

\begin{proof}
  Let $\sigma$ be decomposed as in Lemma \ref{L:rep_dec} and define
  $A = \sigma_1(z) \in B(\mathcal{H}_1)$ and $B = \sigma_2(z) \in B(\mathcal{H}_2)$ and $X = \varphi(1) \in B(\mathcal{H}_1,\mathcal{H}_2)$.
  Then
  \begin{equation*}
    B^* X A = \varphi(\overline{z} 1 z) = \varphi(1) =  X.
  \end{equation*}
  This operator equation has been well studied in the context of Toeplitz operators; see, for instance,
  \cite{Douglas69}.
  We borrow an argument from there.
  By Lemma \ref{lem:unitary_cnu_dec}, we can write $A = A_0 \oplus U$ for a unitary operator
  $U$ and an operator $A_0$ with $\lim_{n \to \infty} A_0^n = 0$.
  Let $U$ act on $M_1 \subset \mathcal{H}_1$.
  Similarly, decompose $B = B_0 \oplus V$, where $V$ is unitary on $M_2 \subset \mathcal{H}_2$.
  Inductively,
  \begin{equation*}
    (B^*)^n X A^n = X.
  \end{equation*}
  Thus, if $\xi \in M_1^\bot$, then
  \begin{equation*}
    \|X \xi\| = \|(B^*)^n X A^n \xi\| = \|(B^*)^n X A_0^n \xi\| \le \|A_0^n \xi\| \xrightarrow{n \to \infty} 0,
  \end{equation*}
  so $X \xi = 0$.
  This implies that
  $X P_{M_1} = X$.
  Applying the same argument to the equation $X^* = A^* X^* B$, we find that
  $P_{M_2} X = X$.
  Therefore
  \begin{equation}
    \label{eqn:red_intertwine}
    X = X P_{M_1} = P_{M_2} X. 
  \end{equation}

  Let $M = M_1 \oplus M_2 \subset \mathcal{H}_1 \oplus \mathcal{H}_2$.
  To see that $M$ is reducing, it suffices to show that $P_M$ commutes
  with $\sigma(a)$ for all $a$ in the dense subset
  \begin{equation*}
    \left\{
      \begin{bmatrix}
        f & 0 \\
        h_1 + \overline{h_2} & \overline{g}
      \end{bmatrix}:
    f,g,h_1,h_2 \in \mathbb{C}[z] \right\}
  \end{equation*}
  of $\mathcal{B}$. So let $f,g, h_1,h_2$ be polynomials.
  Then using the identity for $\varphi$ in Lemma \ref{L:rep_dec} as well as multiplicativity of $\sigma_1$ and $\sigma_2$, we find that
  \begin{equation}
    \label{eqn:red_split}
    \begin{split}
    \sigma \left(
      \begin{bmatrix}
        f & 0 \\ h_1 + \overline{h_2} & \overline{g}
    \end{bmatrix} \right)
      &=
      \begin{bmatrix}
        \sigma_1(f) & 0 \\ \varphi(h_1 + \overline{h_2}) & \sigma_2(g)^*
      \end{bmatrix} \\
      &=
      \begin{bmatrix}
        f(A) & 0 \\
        X h_1(A) + h_2(B)^* X & g(B)^*
      \end{bmatrix}.
    \end{split}
  \end{equation}
  Using that $P_{M_1}$ commutes with $A$, $P_{M_2}$ commutes with $B$ (and hence with $B^*)$
  and \eqref{eqn:red_intertwine}, it follows that $P_M = P_{M_1} \oplus P_{M_2}$
  commutes with the operator matrix in \eqref{eqn:red_split}. Hence $M$ is reducing.

  Moreover, we see from the $(2,1)$-entry in \eqref{eqn:red_split} and \eqref{eqn:red_intertwine} that
  \begin{equation*}
    \varphi(h) = \varphi(h) P_{M_1} = P_{M_2} \varphi(h) \quad \text{ for all } h \in C(\mathbb{T}),
  \end{equation*}
  so
  $\varphi(h)$ maps $M_1$ into $M_2$ and is zero on $M_1^\bot$.
  The decomposition of $\sigma_M$ is given by
  \begin{equation*}
    \sigma_M \left(
      \begin{bmatrix}
        f & 0 \\
        h & \overline{g}
    \end{bmatrix} \right)
      =
    \begin{bmatrix}
      \sigma_1(f) \big|_{M_1}  & 0 \\
      \varphi(h) \big|_{M_1} & \sigma_2(g)^* \big|_{M_2}
    \end{bmatrix}.
  \end{equation*}
  In particular,
  $\sigma_1(z) \big|_{M_1} = A \big|_{M_1} = U$ is unitary.
  Similarly, $\sigma_2(z) \big|_{M_2} = B \big|_{M_2} = V$ is unitary.
  Hence, $\sigma_M$ is of unitary type.
  Finally, since $\varphi(h)$ is zero on $M_1^\bot$ for all $h \in C(\mathbb{T})$, it follows that $\sigma_{M^\bot}$ is of diagonal type.
\end{proof}

Note that the decomposition into unitary and diagonal type is not unique. Indeed, representations can be simultaneously unitary and diagonal type.

We are now ready to prove Theorem \ref{thm:Toeplitz_no_approx}.

\begin{proof}[Proof of Theorem \ref{thm:Toeplitz_no_approx}]
  Let $\sigma: \mathcal{B} \to B(H^2 \oplus H^2)$ be a finite dimensional representation.
  We claim that
  \begin{equation}
    \label{eqn:SOT_closed}
    \left[ \sigma\left(
      \begin{bmatrix}
        0 & 0 \\ 0 &1
    \end{bmatrix} \right)
      -
      \sigma\left(
        \begin{bmatrix}
          0 & 0 \\ 0 & \overline{z}
      \end{bmatrix} \right)^*
      \sigma\left(
        \begin{bmatrix}
          0 & 0 \\ 0 & \overline{z}
    \end{bmatrix} \right) \right]
    \sigma\left(
      \begin{bmatrix}
        0 & 0 \\ 1 & 0
      \end{bmatrix}
    \right) = 0.
  \end{equation}

  To prove this claim, let $\mathcal{H}_0 = \sigma(1) (H^2 \oplus H^2)$, which is finite dimensional. Then $\sigma$ induces a unital representation
  $\widetilde{\sigma}: \mathcal{B} \to B(\mathcal{H}_0)$.
  Lemma \ref{lem:rep_dec} implies that there is a reducing subspace $M \subset \mathcal{H}_0 \subset H^2 \oplus H^2$ for $\sigma$ such that
  the representation $\widetilde{\sigma}_M$ is unitary type and $\widetilde{\sigma}_{M^\bot}$ is diagonal type.
  Since $\widetilde{\sigma}_M$ is unitary type, the first factor in \eqref{eqn:SOT_closed} is zero for $\widetilde{\sigma}_M$ in place of $\sigma$.
  Since $\widetilde{\sigma}_{M^\bot}$ is diagonal type, the second factor in \eqref{eqn:SOT_closed} is zero for $\widetilde{\sigma}_{M^\bot}$
  in place of $\sigma$. Hence the expression in \eqref{eqn:SOT_closed} is zero for $\widetilde{\sigma}$ and hence also for $\sigma$.
  This proves the claim.

  Notice that \eqref{eqn:SOT_closed} is equivalent to
  \begin{equation}
    \label{eqn:SOT_closed_2}
    \sigma\left(
      \begin{bmatrix}
        0 & 0 \\ 1 &0
    \end{bmatrix} \right)
      -
      \sigma\left(
        \begin{bmatrix}
          0 & 0 \\ 0 & \overline{z}
      \end{bmatrix} \right)^*
      \sigma\left(
        \begin{bmatrix}
        0 & 0 \\ \overline{z} & 0
      \end{bmatrix}
      \right)
    = 0.
  \end{equation}
  Recall that if $(T_\lambda)$ and $(S_\lambda)$ are bounded nets in $B(\mathcal{H})$ such that
  $T_\lambda \to T$ in SOT and $S_\lambda \to S$ in WOT, then $S_\lambda T_\lambda \to S T$ in WOT.
  Indeed, if $x,y \in \mathcal{H}$, then
  \begin{align*}
    |\langle (S_\lambda T_\lambda - S T) x, y \rangle |
    &\le | \langle S_\lambda (T_\lambda - T) x , y \rangle | +
     | \langle (S_\lambda - S) T x, y \rangle| \\
    &\le (\sup_\lambda \|S_\lambda\|) \|(T_\lambda - T) x\| \|y \| +
    | \langle (S_\lambda - S) T x, y \rangle| \xrightarrow{\lambda} 0.
  \end{align*}
  From this fact, it follows that the set of all representations satisfying \eqref{eqn:SOT_closed_2} (or equivalently \eqref{eqn:SOT_closed}) is closed
  in the point SOT.

  Finally, notice that the Toeplitz representation $\pi$ does not satisfy \eqref{eqn:SOT_closed_2}, as
  \begin{align*}
    \pi\left(
      \begin{bmatrix}
        0 & 0 \\ 1 &0
    \end{bmatrix} \right)
      -
      \pi\left(
        \begin{bmatrix}
          0 & 0 \\ 0 & \overline{z}
      \end{bmatrix} \right)^*
      \pi\left(
        \begin{bmatrix}
        0 & 0 \\ \overline{z} & 0
      \end{bmatrix}
      \right)
    =
    \begin{bmatrix}
      0 & 0 \\
      I - T_z T_z^* & 0
    \end{bmatrix},
  \end{align*}
  which is not zero since $I - T_z T_z^*$ is the orthogonal projection in $H^2$ onto the constant functions.
  Therefore, $\pi$ cannot be approximated point SOT by finite dimensional representations.
\end{proof}

\bibliographystyle{amsplain}
\bibliography{literature.bib}

\end{document}